\documentclass[11pt,reqno,english]{amsart}

\usepackage{babel}
\usepackage{latexsym,amsfonts}
\usepackage{amsmath}
\usepackage{paralist}
\usepackage{amsthm}

\title[Rainwater--Simons theorems for generalized convergence methods]{Rainwater--Simons type convergence theorems for generalized convergence methods}
\author{Jan-David Hardtke}
\date{}

\providecommand{\cl}[1]{\overline{#1}}
\providecommand{\wscl}[1]{\overline{#1}^{*}}

\DeclareMathOperator{\ex}{ex}
\DeclareMathOperator{\co}{co}
\DeclareMathOperator{\cco}{\cl{\co}}
\DeclareMathOperator{\wscco}{\wscl{\co}}

\newtheorem{theorem}{Theorem}[section]
\newtheorem{corollary}[theorem]{Corollary}

\usepackage{color}
\definecolor{darkgreen}{rgb}{0,0.5,0}
\usepackage[colorlinks,linkcolor=blue,citecolor=red,urlcolor=darkgreen]{hyperref}

\begin{document}

\maketitle

\begin{abstract}
 We extend the well-known Rainwater--Simons convergence theorem to various generalized convergence methods such as strong matrix summability, statistical convergence and almost convergence. In fact we prove these theorems not only for boundaries but for the more general notion of (I)-generating sets introduced by Fonf and Lindenstrauss.
\end{abstract}

\ \par \noindent {\scriptsize {\bf Keywords:} 
Boundary; (I)-generating sets; Rainwater's theorem; Simons' equality; strong matrix summability; statistical convergence; statistically pre-Cauchy sequence; almost convergence 
 
}

\section{Introduction}

First let us fix some notation: throughout this paper we denote by $X$ a Banach space, by $X^{*}$ its dual and by $B_{X}$ its closed unit ball. If $C$ is a convex subset of $X$, then $\ex C$ denotes the set of extreme points of $C$. We write $\co A$ for the convex hull of a subset $A$ of $X$ and $\cl{A}$ for its closure in the norm topology. Finally, for a subset $B$ of $X^{*}$ we denote by $\wscl{B}$ its weak*-closure. \par 

Now recall the notion of boundary: if $X$ is defined over the real field and $K$ is a weak*-compact convex subset of $X^{*}$, a subset $B$ of $K$ is said to be a boundary for $K$ provided that for every $x\in X$ there exists a functional $b\in B$ with $b(x)=\sup _{x^{*}\in K} x^{*}(x)$. In case $K=B_{X^{*}}$ this means that every element of $X$ attains its norm on some functional in $B$. Then $B$ is simply called a boundary for $X$. \par 

It easily follows from the Krein--Milman theorem that $\ex B_{X^{*}}$ is always a boundary for $X$. Rainwater's theorem (cf. \cite{rainwater}) states that a bounded sequence $(x_n)_{n\in \mathbb{N}}$ in $X$ is weakly convergent to $x\in X$ if it is merely convergent to $x$ under every functional $x^{*}\in \ex B_{X^{*}}$. The proof is an application of the Choquet--Bishop--de--Leeuw  theorem (cf. \cite{phelps}) combined with Lebesgue's dominated convergence theorem. In \cite{simons1} and \cite{simons2} Simons has proved a generalization of Rainwater's theorem to arbitrary boundaries. In fact he even proved a stronger statement, namely the following

\begin{theorem}[Simons, cf. \cite{simons2}] \label{th:SLS}
 If $B$ is a boundary for the weak*-compact convex subset $K\subseteq X^{*}$ then 
\begin{equation} 
\sup_{x^{*}\in K} \limsup _{n \to \infty} x^{*}(x_n)=\sup_{x^{*}\in B} \limsup_{n \to \infty} x^{*}(x_n) \label{eq:SLS}
\end{equation}
holds for every bounded sequence $(x_n)_{n\in \mathbb{N}}$ in $X$.
\end{theorem}

\noindent The equality \eqref{eq:SLS} is often referred to as Simons' equality. The proof given in \cite{simons2} is based on an eigenvector argument and actually works for boundaries of arbitrary subsets of $X^{*}$ (which not even need to be weak*-compact or convex) but we are only interested in the above special case. For an easy direct proof of Simons' equality see also \cite[Theorem 2.2]{oja}. From Theorem \ref{th:SLS} it is clear that Rainwater's theorem  holds true for every boundary $B$ of the space $X$. \par 

Next we recall the definition of (I)-generating sets given by Fonf and Lindenstrauss in \cite{fonf1} (here $X$ can be a real or complex space): let $K$ be a weak*-compact convex subset of $X^{*}$ and $B\subseteq K$. Then $B$ is said to (I)-generate $K$ provided that whenever $B$ is written as a countable union $B=\bigcup _{n=1}^{\infty}B_n$ we have that $K=\cco{\bigcup _{n=1}^{\infty} \wscco{B_n}}$. \par 

We clearly have
\begin{equation*}
\cco{B}=K \Rightarrow B \ \text{(I)-generates} \ K \Rightarrow \wscco{B}=K,
\end{equation*} 
but none of the converses is true in general (cf. the examples in \cite{fonf1}). Further note that $B$ (I)-generates $K$ if and only if the following holds: whenever $B$ is written as an increasing union of countably many subsets $(B_n)_{n\in \mathbb{N}}$, then $\bigcup _{n=1}^{\infty} \wscco{B_n}$ is norm-dense in $K$. \par 

Now if $B$ is a boundary for $K$ it follows from the Hahn--Banach separation theorem that $\wscco{B}=K$. In \cite[Theorem 2.3]{fonf1} it is proved that $B$ actually (I)-generates $K$. Together with the observation that for a norm separable (I)-generating subset $B$ of $K$ we already have $\cco{B}=K$ (cf. \cite[Proposition 2.2, (a)]{fonf1}), this leads to a proof of James' celebrated compactness theorem in the separable case (cf. \cite[Theorems 5.7 and 5.9]{fonf2} or the introduction of \cite{kalenda1}). In fact one even gets stronger versions of James' theorem for separable spaces (cf. \cite{fonf1}). \par 

In \cite{nygaard} Nygaard proved that the statement of the Rainwater--Simons convergence theorem holds for every set $B$ that (I)-generates $B_{X^{*}}$ and used this observation combined with the result of Fonf and Lindenstrauss to give a short proof of James' reflexivity criterion in case $B_{X^{*}}$ is weak*-sequentially compact. Independently in \cite{kalenda1} Kalenda introduced the concept of (I)-envelopes and studied the possibility of proving the general James' compactness theorem by these methods. The studies were continued in \cite{kalenda2}. In particular he implicitly proved that the (I)-generation property is equivalent to Simons' equality (cf. \cite[Lemma 2.1]{kalenda1}). Another, explicit proof of this fact may be found in \cite[Theorem 2.2]{cascales}. \par 

Before we can extend the Rainwater--Simons convergence theorem, we have to discuss some generalized convergence methods . This is done in the next section.

\section{Generalized convergence methods}

Consider an infinite complex matrix $A=(a_{nk})_{n,k\in \mathbb{N}}$. A sequence $(s_k)_{k\in \mathbb{N}}$ of scalars is said to be $A$-convergent (or $A$-summable) to $s$, if the series $\sum _{k=1}^{\infty} a_{nk}s_k$ is convergent for every $n\in \mathbb{N}$ and $\lim _{n\to \infty} \sum _{k=1}^{\infty} a_{nk}s_k=s$. \par 

The matrix $A$ is called regular if every sequence which is convergent in the ordinary sense is also $A$-convergent to the same limit. According to a well-known theorem of Toeplitz $A$ is regular if and only if the following conditions are satisfied:
\begin{equation*}
 \sup _{n\in \mathbb{N}} \sum _{k=1}^{\infty} |a_{nk}| < \infty, \  \lim_{n\to \infty} \sum _{k=1}^{\infty} a_{nk}=1 \ \text{and} \ \lim _{n\to \infty} a_{nk}=0 \ \forall k\in \mathbb{N}.
\end{equation*} 
The most prominent example of a regular matrix is the Ces\`aro matrix $C=(c_{nk})_{n,k\in \mathbb{N}}$ defined by $c_{nk}=1/n$ for $k\leq n$ and $c_{nk}=0$ for $k>n$. We refer the reader to \cite{zeller} for more information on regular summability matrices. It is clear that the Rainwater--Simons convergence theorem carries over to matrix summability methods, but less evident that it also holds for the following methods. \par 

If $A$ is a regular positive matrix (i.e., $a_{nk}\geq 0$ for all $n,k\in \mathbb{N}$) and $p>0$, then the sequence $(s_k)_{k\in \mathbb{N}}$ is said to be strongly $A$-$p$-convergent to $s$ provided that $\sum _{k=1}^{\infty} a_{nk}|s_k-s|^p<\infty$ for each $n$ and $\lim _{n\to \infty} \sum_{k=1}^{\infty} a_{nk}|s_k-s|^p=0$. The strong $A$-$p$-convergence is a linear consistent summability method and the strong $A$-$p$-limit of a sequence is unique if it exists. For some results on strong matrix summability we refer to \cite{zeller} (with index $p=1$) or \cite{hamilton}. \par 

In \cite{maddox} Maddox introduced and studied a more general form of strong matrix summability, replacing the index $p$ by a sequence of indices: if $A$ is a positive infinite  matrix and $\mathbf{p}=(p_k)_{k\in \mathbb{N}}$ a sequence of strictly positive numbers, then the sequence $(s_k)_{k\in \mathbb{N}}$ is said to be strongly $A$-$\mathbf{p}$-convergent to $s$ if $\sum _{k=1}^{\infty} a_{nk}|s_k-s|^{p_k}<\infty$ for every $n\in \mathbb{N}$ and $\lim _{n\to \infty} \sum _{k=1}^{\infty} a_{nk}|s_k-s|^{p_k}=0$. Again $A$-$\mathbf{p}$-convergence is a linear method, provided the sequence $\mathbf{p}$ is bounded.\par 

Another common generalized convergence method is that of statistical convergence introduced by Fast in \cite{fast}: a sequence $(s_k)_{k\in \mathbb{N}}$ of (real or complex) numbers is called statistically convergent to $s$ if for each $\varepsilon>0$ we have that $\lim_{n\to \infty} 1/n|\{k\leq n: |s_k-s|\geq \varepsilon\}|=0$. More generally one can consider $A$-statistical convergence for a positive regular matrix $A$: the sequence $(s_k)_{k\in \mathbb{N}}$ is $A$-statistically convergent to $s$ if for each $\varepsilon>0$  we have $\lim_{n\to \infty} \sum _{k=1}^{\infty} a_{nk}\chi_{B_{\varepsilon}}(k)=0$, where $B_{\varepsilon}=\{k\in \mathbb{N}:|s_k-s|\geq \varepsilon\}$ and for $M\subseteq \mathbb{N}$ the symbol $\chi_{M}$ denotes the characteristic function of $M$. For $A=C$, the Ces\`aro matrix, we have the ordinary statistical convergence. It is easy to check that $A$-statistical convergence is a linear consistent method and that the $A$-statistical limit is uniquely determined whenever it exists. In \cite{connor1} Connor proved the following connection between statistical and strong Ces\`aro convergence:

\begin{theorem}[Connor, cf. \cite{connor1}] \label{th:stat-str-matrix-conv}
Let $(s_k)_{k\in \mathbb{N}}$ be a sequence of numbers, $p>0$ and $s$ a number. Then the following hold
\begin{enumerate}[\upshape(i)]
\item If $(s_k)_{k\in \mathbb{N}}$ is strongly $p$-Ces\`aro convergent to $s$, then it is also statistically convergent to $s$.
\item If $(s_k)_{k\in \mathbb{N}}$ is bounded and statistically convergent to $s$, then it is also strongly $p$-Ces\`aro convergent to $s$.
\end{enumerate}
\end{theorem}

\noindent Virtually the same proof as given in \cite{connor1} also works for $A$-statistical and strong $A$-$p$-convergence in case of an arbitrary positive regular matrix $A$. In particular, $A$-statistical and strong $A$-$p$-convergence are equivalent on bounded sequences (for this see also \cite[Theorem 8]{connor2}) and hence for any two indices $p,q>0$ strong $A$-$p$- and strong $A$-$q$-convergence are equivalent on bounded sequences. \par

 We further recall the notion of statistically pre-Cauchy sequences, introduced in \cite{connor4}: a sequence $(s_k)_{k\in \mathbb{N}}$ of scalars is called statistically pre-Cauchy if $\lim_{n\to \infty} 1/n^2|\{(i,j)\in \{1,\dots,n\}^2: |s_i-s_j|\geq \varepsilon\}|=0$ for all $\varepsilon>0$. It is proved in \cite{connor4} that a statistically convergent sequence is statistically pre-Cauchy, whereas the converse is not true in general, but under certain additional assumptions (cf. \cite[Theorems 5 and 7]{connor4}). Also, the following analogue of theorem 2.1 holds:

\begin{theorem}[Connor et al., cf. \cite{connor4}] \label{th:stat-pre-cauchy}
A sequence $(s_k)_{k\in \mathbb{N}}$ is statistically pre-Cauchy if 
\begin{equation}
\lim_{n\to \infty} \frac{1}{n^2} \sum _{i,j\leq n} |s_i-s_j|=0. \label{eq:stat-pre-cauchy}
\end{equation}
The converse is true if $(s_k)_{k\in \mathbb{N}}$ is bounded.
\end{theorem}

 It is easy to deduce from the classical Rainwater--Simons convergence theorem the fact that a bounded sequence in $X$ is weakly Cauchy if and only if it is a Cauchy sequence under every functional in $B$, where $B$ is any boundary for $X$. In section 3 we shall see that the same statement holds if one replaces ``Cauchy sequence'' by ``statistically pre-Cauchy sequence'' and $B$ is any (I)-generating subset of $B_{X^{*}}$. \par 

We remark that there exists another notion of statistically Cauchy sequences introduced in \cite{fridy}, which turns out to be equivalent to statistical convergence (cf. \cite[Theorem 1]{fridy}), but this criterion is difficult to apply if one has no idea what the statistical limit might look like. This was the main motivation for introducing the concept of statistically pre-Cauchy sequences in \cite{connor4}. \par 

More information on statistical convergence can be found in \cite{connor1}, \cite{connor2} and \cite{fridy}. For some applications of statistical convergence in Banach space theory see also \cite{connor3}. \par 

Finally, let us discuss the notion of almost convergence. For this we first recall the definition of a Banach limit: if $L:\ell^{\infty} \rightarrow \mathbb{R}$ is a linear functional with $L(\mathbf{1})=1$, $\mathbf{x}\geq 0 \Rightarrow L(\mathbf{x})\geq 0$ and $L(T\mathbf{x})=L(\mathbf{x})$ for each $\mathbf{x}\in \ell^{\infty}$, where $\mathbf{1}=(1,1,\dots)$ and $T:\ell^{\infty} \rightarrow \ell^{\infty}$ denotes the shift operator (i.e., $(T\mathbf{x})(n)=\mathbf{x}(n+1)$), then $L$ is called a Banach limit. The existence of a Banach limit can be easily proved using the Hahn--Banach extension theorem. \par 

In \cite{lorentz} Lorentz defined a bounded sequence $(s_k)_{k\in \mathbb{N}}$ of real numbers to be almost convergent to $s\in \mathbb{R}$ if $L(s_k)=s$ for every Banach limit $L$. It is easy to see that every convergent sequence is also almost convergent (to the same limit). For an easy example showing that the converse is not true, note that the sequence $(1,0,1,0,\dots)$ is almost convergent to $1/2$. In general Lorentz proved that almost convergence is equivalent to ``uniform Ces\`aro convergence'' in the following sense:

\begin{theorem}[Lorentz, cf. \cite{lorentz}] \label{th:alm-conv}
 A bounded sequence $(s_k)_{k\in \mathbb{N}}$ of real numbers is almost convergent to $s\in \mathbb{R}$ if and only if
\begin{equation}
 \frac{1}{n} \sum _{k=1}^{n} s_{k+l} \xrightarrow{n\to \infty} s \  \text{uniformly in} \ l\in \mathbb{N}_0. \label{eq:alm-conv}
\end{equation}
\end{theorem}

\noindent Lorentz then introduced the notion of $F_A$-convergence, replacing the Ces\`aro matrix in (2) by an arbitrary regular matrix $A$: a bounded sequence $(s_k)_{k\in \mathbb{N}}$ is said to be $F_A$-convergent to $s$ if 
\begin{equation*}
\sum _{k=1}^{\infty} a_{nk}s_{k+l} \xrightarrow{n\to \infty} s \ \text{uniformly in} \ l\in \mathbb{N}_0.
\end{equation*}
In particular, Lorentz characterized those regular matrices $A$ for which $F_A$- and almost convergence are equivalent. We refer to \cite{lorentz} for information on this subject. Further references for generalized convergence methods can be found in the literature mentioned above.

\section{Extending the convergence theorem}

We now prove a general theorem, resembling \ref{th:SLS}, from which the extended forms of the convergence theorem will easily follow. We denote by $\tau_p$ the topology of pointwise convergence on $\ell^{\infty}$.

\begin{theorem} \label{th:main thm}
 Let $K$ be a weak*-compact convex subset of $X^{*}$ and $B$ an (I)-generating subset of $K$. Further, let $P:\ell^{\infty} \rightarrow \ell^{\infty}$ be a map with $P(0)=0$. Denote by $P_n$ the map $\mathbf{x} \mapsto |(P\mathbf{x})(n)|$ and suppose that the following conditions are satisfied: 
\begin{enumerate}[\upshape(i)]
\item For each $n$ the map $P_n$ is convex and lower semicontinuous with respect to $\tau_p$ on every bounded subset of $\ell^{\infty}$.
\item There exists $M\geq 0$ with $P_n(\mathbf{x}+\mathbf{y}) \leq M(P_n\mathbf{x}+P_n\mathbf{y})$ for all $n\in \mathbb{N}$ and all $\mathbf{x},\mathbf{y} \in \ell^{\infty}$.
\item $P$ is continuous at $0$ with respect to the norm topology of $\ell^{\infty}$.
\end{enumerate}
Then for every bounded sequence $\mathbf{x}=(x_n)_{n\in \mathbb{N}}$ in $X$ we have
\begin{equation}
\sup _{x^{*}\in K} \limsup _{n\to \infty} P_n(x^{*}(\mathbf{x})) \leq M\sup _{x^{*}\in B} \limsup _{n\to \infty} \label{eq:main thm} P_n(x^{*}(\mathbf{x})) \ ,
\end{equation} 
where $x^{*}(\mathbf{x})$ denotes the sequence $(x^{*}(x_n))_{n\in \mathbb{N}}$.
\end{theorem}

\begin{proof}
Denote the supremum on the right hand side of (4) by $S$. If $S=\infty$ the statement is clear, so we may assume $S<\infty$. Now take $x^{*}\in K$ and $\varepsilon>0$ and fix a constant $R>0$ with $\|x_n\|\leq R$ for all $n$. Define for all $N\in \mathbb{N}$
\begin{equation*}
B_N=\{y^{*}\in B: P_n(y^{*}(\mathbf{x}))\leq S+\varepsilon \ \forall n\geq N\}.
\end{equation*} 
Then $B_N\nearrow B$ and since $B$ (I)-generates $K$ it follows that 
\begin{equation}
\cl{\bigcup _{N=1}^{\infty} \wscco{B_N}}=K. \label{eq:1}
\end{equation}
By (iii) we can find $\delta>0$ such that for all $\mathbf{y} \in \ell^{\infty}$ we have 
\begin{equation}
\|\mathbf{y}\|_{\infty}\leq \delta \ \Rightarrow \ \|P\mathbf{y}\|_{\infty}\leq \varepsilon. \label{eq:2}
\end{equation}
By \eqref{eq:1} there exists an index $N\in \mathbb{N}$ and a functional $\tilde{x}^{*}\in \wscco{B_N}$ with $\|x^{*}-\tilde{x}^{*}\|\leq \delta / R$. From (i) and the definition of $B_N$ we conclude that 
\begin{equation}
P_n(\tilde{x}^{*}(\mathbf{x}))\leq S+\varepsilon \ \ \forall n\geq N. \label{eq:3}
\end{equation}
Now for all $n\geq N$ we have 
\begin{align*}
&P_n(x^{*}(\mathbf{x})) \stackrel{(ii)}{\leq} M(P_n(x^{*}(\mathbf{x})-\tilde{x}^{*}(\mathbf{x}))+P_n(\tilde{x}^{*}(\mathbf{x}))) \\ &\stackrel{\eqref{eq:3}}{\leq} M(P_n(x^{*}(\mathbf{x})-\tilde{x}^{*}(\mathbf{x}))+S+\varepsilon) \leq M(S+2\varepsilon) \ ,
\end{align*}
where the last inequality holds because of $\|x^{*}(\mathbf{x})-\tilde{x}^{*}(\mathbf{x})\|_{\infty}\leq \|x^{*}-\tilde{x}^{*}\|R\leq \delta$ and \eqref{eq:2}. So we can conclude $\limsup_{n\to \infty} P_n(x^{*}(\mathbf{x})) \leq M(S+2\varepsilon)$ and since $\varepsilon$ was arbitrary the proof is finished.
\end{proof}

\noindent We remark that the above proof is a slight modification of the argument for the first implication in the proof of \cite[Theorem 2.2]{cascales}. \par 

We can now collect some corollaries. First we consider strong matrix summability and related methods as described in the previous section.

\begin{corollary} \label{cor:str-matrix SLS1}
Let $B$ be an (I)-generating subset of the weak*-compact convex set $K\subseteq X^{*}$, $A=(a_{nk})_{n,k\in \mathbb{N}}$ a positive regular matrix and $\mathbf{p}=(p_k)_{k\in \mathbb{N}}$ a sequence in $\mathbb{R}$ with $p_k\geq 1$ for all $k$ and $r=\sup _{k\in \mathbb{N}} p_k < \infty$. \par 
\noindent Then for every bounded sequence $(x_n)_{n\in \mathbb{N}}$ in $X$ we have 
\begin{equation}
\sup_{x^{*}\in K} \limsup_{n\to \infty} \sum_{k=1}^{\infty} a_{nk}|x^{*}(x_n)|^{p_k} \leq 2^{r-1}\sup_{x^{*}\in B} \limsup_{n\to \infty} \sum_{k=1}^{\infty} a_{nk} |x^{*}(x_n)|^{p_k}.
\end{equation}
\end{corollary}

\begin{proof}
Define $P:\ell^{\infty} \rightarrow \ell^{\infty}$ by $(P\mathbf{x})(n)=\sum_{k=1}^{\infty}a_{nk}|\mathbf{x}(k)|^{p_k}$ for all $n\in \mathbb{N}$ and all $\mathbf{x}\in \ell^{\infty}$. Since for each $p\geq 1$ the function $t \mapsto t^p$ is convex, it follows that the map $P$ is coordinatewise convex. Moreover, it is easy to see that each coordinate function of $P$ is actually continuous with respect to $\tau_p$ on every bounded subset of $\ell^{\infty}$, thus $P$ satisfies the condition (i) of theorem 3.1. The condition (iii) is easily seen to be fulfilled as well. Finally, because of the convexity of $t \mapsto t^p$ for $p\geq 1$, we have $|a+b|^p\leq 2^{p-1}(|a|^p+|b|^p)$ for all $a,b\in \mathbb{C}$ and all $p\geq 1$ and it follows that $P$ also satisfies the conditon (ii) with $M=2^{r-1}$. Theorem \ref{th:main thm} now yields the desired inequality.
\end{proof}

 For a constant sequence $\mathbf{p}$ even more is true.

\begin{corollary} \label{cor:str-matrix SLS2}
Let $K$, $B$ and $A$ be as in Corollary \ref{cor:str-matrix SLS1}. Then for each $p\geq 1$ and every bounded sequence $(x_n)_{n\in \mathbb{N}}$ in $X$, the following equality holds:
\begin{equation}
\sup _{x^{*}\in K} \limsup_{n\to \infty} \sum_{k=1}^{\infty} a_{nk}|x^{*}(x_n)|^p=\sup _{x^{*}\in B} \limsup _{n\to \infty} \sum_{k=1}^{\infty} a_{nk}|x^{*}(x_n)|^p. \label{eq:str-matrix SLS2}
\end{equation}
\end{corollary}
\begin{proof}
This time we define $P:\ell^{\infty} \rightarrow \ell^{\infty}$ by 
\begin{equation*}
(P\mathbf{x})(n)=\left(\sum_{k=1}^{\infty} a_{nk}|\mathbf{x}(k)|^p\right)^{1/p} \ \forall n\in \mathbb{N}, \forall \mathbf{x}\in \ell^{\infty}.
\end{equation*}
The Minkowski inequality implies that $P$ fulfils (ii) with $M=1$ and conditions (i) and (iii) are fulfilled as well, so \eqref{eq:str-matrix SLS2} follows from Theorem \ref{th:main thm}.
\end{proof}

Now we can extend the Rainwater--Simons convergence theorem to strong matrix summability methods (and en passant also to statistical convergence).

\begin{corollary} \label{cor:RS-thm str-matrix stat conv}
Let $A=(a_{nk})_{n,k\in \mathbb{N}}$ be a positive regular matrix, $B$ an (I)-generating subset of $B_{X^{*}}$ and $\mathbf{p}=(p_k)_{k\in \mathbb{N}}$ a sequence of real numbers with $q=\inf_{k\in \mathbb{N}}p_k >0$ and $r=\sup_{k\in \mathbb{N}}p_k <\infty$. \par 
\noindent Then a bounded sequence $(x_n)_{n\in \mathbb{N}}$ in $X$ is strongly $A$-$\mathbf{p}$-convergent to $x\in X$ under every functional $x^{*}\in X^{*}$ if (and only if) it is strongly $A$-$\mathbf{p}$-convergent to $x$ under every functional in $B$. \par 
\noindent The same statement also holds for $A$-statistical convergence.
\end{corollary}

\begin{proof}
In case $q\geq 1$ this is immediate from Corollary \ref{cor:str-matrix SLS1}. In general we can simply replace the sequence $\mathbf{p}$ by $(p_k/q)_{k\in \mathbb{N}}$ and get the same result, because for bounded sequences strong $A$-$p$- and strong $A$-$s$-convergence are equivalent for any $p,s>0$ by the remark following Theorem \ref{th:stat-str-matrix-conv}. The case of $A$-statistical convergence also follows from this remark.
\end{proof}

The case of statistically pre-Cauchy sequences can be treated in much the same way.

\begin{corollary} \label{cor:RS-thm stat-pre-Cauchy}
Let $B$ be an (I)-generating subset of $B_{X^{*}}$ and $(x_n)_{n\in \mathbb{N}}$ a bounded sequence in $X$ such that $(x^{*}(x_n))_{n\in \mathbb{N}}$ is statistically pre-Cauchy for all $x^{*}\in B$. Then $(x_n)_{n\in \mathbb{N}}$ is ``weakly statistically pre-Cauchy'', i.e. $(x^{*}(x_n))_{n\in \mathbb{N}}$ is statistically pre-Cauchy for every $x^{*}\in X^{*}$.
\end{corollary}

\begin{proof}
Define $P:\ell^{\infty} \rightarrow \ell^{\infty}$ by 
\begin{equation*}
(P\mathbf{x})(n)=\frac{1}{n^2} \sum_{i,j\leq n} |\mathbf{x}(i)-\mathbf{x}(j)| \ \ \forall n\in \mathbb{N}, \forall \mathbf{x}\in \ell^{\infty}
\end{equation*}
and apply Theorems \ref{th:stat-pre-cauchy} and \ref{th:main thm} to get the desired conclusion.
\end{proof}

Next we consider the case of $F_A$-convergence.

\begin{corollary} \label{cor:RS-thm FA-conv}
Let $B$ be an (I)-generating subset of the dual unit ball $B_{X^{*}}$ and $A=(a_{nk})_{n,k\in \mathbb{N}}$ a regular matrix. Further, let $(x_n)_{n\in \mathbb{N}}$ be a bounded sequence in $X$ as well as $x\in X$ such that $(x^{*}(x_n))_{n\in \mathbb{N}}$ is $F_A$-convergent to $x^{*}(x)$ for all $x^{*}\in B$. \par 
\noindent Then $(x_n)_{n\in \mathbb{N}}$ is $F_A$-convergent to $x$ under every functional $x^{*}\in X^{*}$. \par 
\noindent In particular this is true for the method of almost convergence.
\end{corollary}

\begin{proof}
We define $P:\ell^{\infty} \rightarrow \ell^{\infty}$ by 
\begin{equation*}
(P\mathbf{x})(n)=\sup _{l\in \mathbb{N}} \left|\sum_{k=1}^{\infty} a_{nk}\mathbf{x}(k+l)\right| \ \forall n\in \mathbb{N},\forall \mathbf{x}\in \ell^{\infty}.
\end{equation*}
Then $P$ fulfils the conditions (i), (ii) and (iii) of Theorem \ref{th:main thm} (with M=1 in (ii)) and so the assertion easily follows. The ``in particular'' part follows from Theorem \ref{th:alm-conv}.
\end{proof}

Let us finish this note with an application of Corollary \ref{cor:RS-thm str-matrix stat conv}. As mentioned before, it is proved in \cite{nygaard} that a Banach space $X$ whose dual unit ball is weak*-sequentially compact is reflexive if $B_X$ (I)-generates $B_{X^{**}}$, see also \cite[Corollaries 3.5, 3.6 and 3.7]{kalenda1}. Further, it is proved in \cite{kalenda2} that every non-reflexive Banach space can be renormed such that the unit ball in this renorming does not (I)-generate the respective bidual unit ball. Corollary \ref{cor:reflexivity} below can be viewed as a slight generalization of the result from \cite{nygaard}. \par 

We remark that if $A$ is a positive regular matrix and $(s_k)_{k\in \mathbb{N}}$ a sequence of non-negative real numbers which is $A$-convergent to zero, then it is easy to see that $0$ is a cluster point of $(s_k)_{k\in \mathbb{N}}$ (in the ordinary sense). From this it is easy to deduce the following: if $(x_n)_{n\in \mathbb{N}}$ is a sequence in $X$ which is strongly $A$-convergent to $x\in X$ under every $x^{*}\in X^{*}$, then $x$ is a weak cluster point of $(x_n)_{n\in \mathbb{N}}$. \par 

\begin{corollary} \label{cor:weakcom}
Suppose that $B$ is an (I)-generating subset of $B_{X^{*}}$ and that $M\subseteq X$ is bounded. If for each sequence $(x_n)_{n\in \mathbb{N}}$ in $M$ there is a positive regular matrix $A$ and an $x\in X$ such that $(x^{*}(x_n))_{n\in \mathbb{N}}$ is strongly $A$-convergent to $x^{*}(x)$ for all $x^{*}\in B$, then $M$ is relatively weakly compact. \par 
\noindent In particular, $M$ is relatively weakly compact if each sequence in $M$ has a subsequence that is statistically convergent to some $x\in X$ under every functional in $B$ (i.e. if $M$ is ``statistically sequentially compact'' in the topology of pointwise convergence on $B$).
\end{corollary}

\begin{proof}
From Corollary \ref{cor:RS-thm str-matrix stat conv} and the above remark we conclude that $M$ is re\-latively weakly countably compact and hence also relatively weakly compact by the Eberlein--Shmulyan theorem.
\end{proof}

From Corollary \ref{cor:weakcom} our reflexivity result now immediately follows.

\begin{corollary} \label{cor:reflexivity}
Suppose that $B_X$ (I)-generates $B_{X^{**}}$ and that for each sequence $(x_n^{*})_{n\in \mathbb{N}}$ in $B_{X^{*}}$ there is a positive regular matrix $A$ and an $x^{*}\in X^{*}$ such that $(x_n^{*}(x))_{n\in \mathbb{N}}$ is strongly $A$-convergent to $x^{*}(x)$ for all $x\in X$. Then $X$ is reflexive.
In particular, $X$ is reflexive if $B_X$ (I)-generates $B_{X^{**}}$ and $B_{X^{*}}$ is ``weak*-statistically sequentially compact''.
\end{corollary}

\begin{proof}
From Corollary \ref{cor:weakcom} it follows that $B_{X^{*}}$ is weakly compact, thus $X^{*}$ and hence also $X$ is reflexive.
\end{proof}

\noindent {\em Remark.} The author was informed by the referee that Corollaries \ref{cor:weakcom} and \ref{cor:reflexivity} are related to some known characterizations of weak compactness, namely the result from \cite{pelczynski} combined with \cite[Theorem 1]{lorentz2} immediately yields the following: a bounded weakly closed subset $K$ of $X$ is weakly compact if and only if for each sequence $(x_n)_{n\in \mathbb{N}}$ in $K$ there is some $x\in X$ and a positive regular row-finite\footnote{This means that every row of the matrix contains only finitely many non-zero elements.} matrix $A$ such that $(x_n)_{n\in \mathbb{N}}$ is strongly $A$-convergent to $x$ under every functional $x^{*}\in X^{*}$. In particular, one gets the following analogue of the result from \cite{singer}: the Banach space $X$ is reflexive if and only if for every bounded sequence $(x_n)_{n\in \mathbb{N}}$ in $X$ there exists some $x\in X$ and a positive regular row-finite matrix $A$ such that $(x^{*}(x_n))_{n\in \mathbb{N}}$ is strongly $A$-summable to $x^{*}(x)$ for every $x^{*}\in X^{*}$. \par

Also, the author was informed by V. Kadets about some related results in the paper \cite{kadets}, where the Rainwater theorem for filter convergence is studied in detail.

\ \par \noindent

{\sc \noindent Department of Mathematics \\ Freie Universit\"at Berlin \\ Arnimallee 6, 14195 Berlin \\ Germany \\} 
{\it E-mail address:} \href{mailto:hardtke@math.fu-berlin.de} {\tt hardtke@math.fu-berlin.de}

\end{document}